\documentclass[a4paper,reqno]{amsart}

 \usepackage{amsmath,amsthm, amssymb}
\usepackage{mathrsfs}
\usepackage[shortlabels]{enumitem}
\usepackage{graphicx}
\usepackage[font=small]{caption}
\usepackage{xcolor}
\usepackage{url}

\newcommand\e{\mathrm e} 

\newcommand{\N}{\mathbb{N}}
\newcommand{\R}{{\mathbb{R}}}
\newcommand{\C}{{\mathbb{C}}}
\newcommand{\Z}{{\mathbb{Z}}}

\newcommand{\dd}{{{\rm d}}}
\newcommand{\ii}{{\rm i}}

\newcommand{\ie}{{\emph{i.e.}}}
\newcommand{\eg}{{\emph{e.g.}}}

\newcommand{\ov}{\overline}

\newcommand\ds{\displaystyle}

\newcommand{\eps}{\varepsilon}

\newcommand{\spd}{\sigma_{\rm disc}}

\newcommand{\se}[1]{\sigma_{\rm e#1}}
\newcommand{\spp}{\sigma_{\rm p}}

\newcommand{\Ran}{{\operatorname{Ran}}}

\renewcommand{\Re}{\operatorname{Re}}
\renewcommand{\Im}{\operatorname{Im}}
\newcommand{\dist}{\operatorname{dist}}
\newcommand{\sgn}{\operatorname{sgn}}

\newcommand{\Tr}{\operatorname{Tr}}

\newcommand{\BigO}{\mathcal{O}}
\newcommand{\diag}{\operatorname{diag}}

\theoremstyle{plain}

\newtheorem{theorem}{Theorem}[section]
\newtheorem{lemma}[theorem]{Lemma}

\newtheorem{corollary}[theorem]{Corollary}

\theoremstyle{definition}
\newtheorem{example}[theorem]{Example}
\newtheorem{remark}[theorem]{Remark}

\newcommand\cA{\mathcal A}
\newcommand\cB{\mathcal B}

\newcommand\cL{\mathcal L}
\newcommand\cM{\mathcal M}
\newcommand\cN{\mathcal N}

\newcommand\cQ{\mathcal Q}
\newcommand\cR{\mathcal R}

\usepackage{mathtools}
\mathtoolsset{showonlyrefs}

\numberwithin{equation}{section}

\usepackage[normalem]{ulem}
\newcommand{\pd}{\textcolor{black}}

\newcommand{\tn}[1]{{\left\vert\kern-0.25ex\left\vert\kern-0.25ex\left\vert #1 
		\right\vert\kern-0.25ex\right\vert\kern-0.25ex\right\vert}}

\begin{document}

\title[Eigenvalues of non-self-adjoint Dirac operators]{Eigenvalues of one-dimensional non-self-adjoint Dirac operators and applications}

\author{Jean-Claude Cuenin}
\address[Jean-Claude Cuenin]{
Mathematisches Institut, Universit\"at M\"unchen,
Theresienstr.~39, D-80333 M\"unchen}
\email{cuenin@math.lmu.de}

\author{Petr Siegl}
\address[Petr Siegl]{
Mathematical Institute, 
University of Bern,
Alpeneggstr.\ 22,
3012 Bern, Switzerland
\& On leave from Nuclear Physics Institute CAS, 25068 \v Re\v z, Czechia}
\email{petr.siegl@math.unibe.ch}



\date{25th April 2017}

\begin{abstract}
We analyze eigenvalues emerging from thresholds of the essential spectrum of one-dimensional Dirac operators perturbed by complex and non-symmetric potentials. In the general non-self-adjoint setting we establish the existence and asymptotics of weakly coupled eigenvalues and Lieb-Thirring inequalities. As physical applications we investigate the damped wave equation and armchair graphene nanoribbons.  
\end{abstract}

\thanks{The research of P.S. is supported by the \emph{Swiss National Science Foundation}, SNF Ambizione grant No. PZ00P2\_154786.
}

\maketitle


\section{Introduction}

Dirac operators attracted considerable attention in recent years, in particular in the context of non-self-adjoint spectral theory \cite{Cuenin-2014-79, Cuenin-2017-272, Cuenin-2014-15, Djakov-2012-263, Elton-2013, Lunyov-2016-441, Savchuk-2014-96}, nonlinear Schr\"odinger equations \eg~\cite{Cascaval-2004-93,Kappeler-2014-359} or as an effective model for graphene \cite{Benguria-2017, Cuenin-2014-55, Freitas-2014-26, Jakubsky-2013-331}. In this paper we analyze eigenvalues emerging from the thresholds of the essential spectrum of the one-dimensional Dirac operator in $L^2(\R)$ perturbed by a general matrix-valued and non-symmetric potential $V$ preserving the essential spectrum. 

Our main results include the existence and asymptotics of weakly coupled eigenvalues for the one dimensional Dirac operator (Theorem~\ref{thm:1D.weak}) and Lieb-Thirring type inequalities (Theorem~\ref{thm.LT}) in the massive as well as the massless case. These results complement the eigenvalue estimates in \cite{Cuenin-2014-15} and also show that the latter are optimal in the weak coupling regime, see Remark~\ref{rem:wc}. 

As physical applications, we investigate the damped wave equation in $L^2(\R)$ and to a two-dimensional model of charge carriers in graphene nanoribbons (or waveguides) with so-called armchair boundary conditions. We emphasize here the inherent  non-self-adjoint nature of the former caused by the presence of damping. Moreover, our eigenvalue estimates may be converted to resonance estimates via the well-known method of complex scaling, as in \cite{Cuenin-2014-15}.

The application for the damped wave equation (Theorem~\ref{thm:DWE}) demonstrates a natural effect from the physical point of view: The integrable damping $\eps a_1(x)$ cannot affect the essential spectrum; however, for any $\eps>0$, it gives rise to a pair of complex conjugated eigenvalues having the tendency to meet at the real axis. The interpretation of the results for the graphene armchair waveguides is more complicated due to the $4\times 4$ matrix structure and the PDE nature of the problem. Nonetheless, in the simplest setting of a diagonal potential that is constant in the transverse direction, the quantities entering the eigenvalue asymptotics are expressed in terms of the integral of the trace of $V$ only, see Example~\ref{Ex:V.diag.ac} and Theorem~\ref{thm:ac.weak}.

The main ingredient in the proofs is the analysis of a Birman-Schwinger operator. Since the problem is not self-adjoint, the existence of eigenvalues in the gap of the essential spectrum does not follow from min-max considerations. Nonetheless, the weak coupling technique, relying on the isolation of a singular part $L$ of the Birman-Schwinger operator, admits a generalization to the non-self-adjoint setting. This is possible since $L$ is of finite rank and so the question of existence and asymptotics of eigenvalues is converted to a matrix problem, which is analyzed with the help of Rouch\'e's theorem eventually.
The proofs of the Lieb-Thirring type bounds are also based on complex analysis techniques, this time on a generalization of Jensen's identity due to \cite{Borichev-2009-41}. Inequalities of this type were established in \cite{Dubuisson-2014-78} for one- and multidimensional Dirac operators, and improved results in the multidimensional case recently appeared in \cite{Cuenin-2017-272}. The difference of our new estimates compared to the one-dimensional results in \cite{Dubuisson-2014-78} is that the weights in the eigenvalue sums are better, which leads to tighter upper bounds for the number of eigenvalues in certain subsets of the complex plane. The price to pay for this improvement is that the eigenvalue sums cannot be controlled by a single $L^p$ norm, but only by a combination of two such norms. This phenomenon was already encountered in \cite{Cuenin-2017-272}, and the reason for it is a lack of decay of the free resolvent as the spectral parameter tends to infinity.

To avoid technicalities related to domain questions we intentionally require that $V$ is both integrable and square-integrable throughout the entire paper. The technique allowing the omission of the $L^2$ assumption is described in~\cite[Sec.~6]{Cuenin-2014-15}. In the waveguide case, the $L^2$ assumption is also convenient, though not essential, when estimating the infinite sums arising from the decomposition of the resolvent, see Remark~\ref{rem:ac.L2}. To simplify the presentation of the weak coupling eigenvalue asymptotics \eqref{z+.exp}, \eqref{z-.exp} and \eqref{z+.exp.ac}, \eqref{z-.exp.ac}, we also do not strive for higher order terms in the expansion, although these could in principle be obtained in a similar way as in the Schr\"odinger case, see \eg~\cite{Simon-1976-97,Klaus-1977-108}.  We have now also all needed ingredients in hand to prove an analogue of the Lieb-Thirring inequalities in Subsection~\ref{subsec:1D.LT} for the graphene waveguides. However, the conformal map would be much more involved and we do not pursue this direction.

The paper is organized as follows. In Section 2 we briefly recall the relevant results of \cite{Cuenin-2014-15} for the one dimensional non-self-adjoint Dirac operator and establish the weak coupling eigenvalue asymptotics (Subsection~\ref{subsec:1D.wc}), and the Lieb-Thirring inequalities (Subsection~\ref{subsec:1D.LT}). 
In Section~\ref{sec:appl}, we apply these results to the one-dimensional damped wave equation,  (Subsection~\ref{subsec:DWE}), and graphene waveguides (Subsection~\ref{subsec:ac}).

\section{One dimensional Dirac operator}
\label{sec:1D}

The spectrum of the free operator $H$ with $m \geq 0$ in $L^2(\R;\C^2)$ 
\begin{equation}\label{H.def}
H =
- \ii \partial_x \sigma_1 + m \sigma_3=
 \begin{pmatrix}
m & - \ii \partial_x
\\
- \ii \partial_x & -m 
\end{pmatrix}
\end{equation}
reads
\begin{equation}\label{H.sp}
\sigma(H) = \se{3}(H) = (-\infty, -m] \cup [m, \infty);
\end{equation}
here and in the sequel, we use the essential spectrum $\se{3}$ defined as 
\begin{equation}
\sigma_{e3}(T):=\{ z \in \C \,: \, T-z \text{ is not Fredholm}\},
\end{equation}
see \eg~\cite[Sec.~IX]{EE} for details, and the Pauli matrices
\begin{equation}
\sigma_1=\begin{pmatrix}
0&1\\1&0
\end{pmatrix},
\quad 
\sigma_2=\begin{pmatrix}
0&-\ii\\\ii&0
\end{pmatrix},
\quad 
\sigma_3=\begin{pmatrix}
1&0\\0&-1
\end{pmatrix}.
\end{equation}
The $n\times n$ identity matrix is denoted by $I_n$.

As a perturbation, we consider a (possibly complex and non-symmetric) matrix potential 
\begin{equation}\label{V.def}
V: \R \to \C^{2 \times 2}, \quad \|V\| \in L^1(\R) \cap L^2(\R),
\end{equation}
where $\|V(x)\|$ is the operator norm in $\C^2$ of the matrix $V(x)$.

Our goal is to analyze the spectrum of $H+ V$.
Note that the $L^2$ condition in \eqref{V.def} is imposed for technical reasons only and is not strictly necessary for the results of this section. It offers the advantage that $H+V$ may be defined as an operator sum because $V$ is relatively $H$-compact and hence infinitesimally $H$-bounded. In addition, the relative compactness implies that the essential spectrum is stable, i.e.\ $\se{3}(H+V)=\se{3}(H)$. If only the $L^1$ condition is assumed in \eqref{V.def}, the perturbed operator can be defined by means of a resolvent formula; we refer to \cite{Cuenin-2014-15} for the details.

For any $p\in[1,\infty)$ we set
\begin{equation}\label{V1.norm.def}
\|V\|_p^p:= \int_{\R} \|V(x)\|^p \; \dd x.
\end{equation}	
%
%
It is proved in \cite[Thm.~2.1]{Cuenin-2014-15} that if $\|V\|_1 <1$, then all non-embedded eigenvalues of $H+V$ satisfy
\begin{equation}\label{spectral inclusion}
\spp(H+V) \setminus \se{3}(H) \subset \ov B_{mr_0}(m x_0) \, \dot{\cup} \, \ov B_{mr_0}(-m x_0) 
\end{equation}
where
\begin{equation}\label{x0r0}
x_0 := \left(\frac{\|V\|_1^4 - 2 \|V\|_1^2 +2}{4(1-\|V\|_1^2)} + \frac 12\right)^\frac12, \quad 
r_0 := \left(\frac{\|V\|_1^4 - 2 \|V\|_1^2 +2}{4(1-\|V\|_1^2)} - \frac 12\right)^\frac12.
\end{equation}

\subsection{Weakly coupled eigenvalues}
\label{subsec:1D.wc}

Here we analyze the point spectrum of $H+\eps V$ as $\eps \to 0+$, i.e.~the weak coupling regime. 
In the self-adjoint setting, a straightforward construction of test functions together with a min-max argument applied to $(H+V)^2$ shows that if the matrix 
\begin{equation}\label{M.cond}
\int_{\R} \left(V(x)^2 + m \{\sigma_3, V(x)\} \right)\; \dd x,
\end{equation}
where the brackets $\{\cdot, \cdot\}$ denote the anticommutator, has a negative eigenvalue, then $H + V$ has an eigenvalue in $(-m,m)$. In the weak coupling regime, \eqref{M.cond} can be translated (by ignoring the term of $\BigO(\eps^2)$) to $\pd{\int_{\R} V_{11} <0}$ or $\pd{\int_{\R} V_{22} > 0}$. In Theorem~\ref{thm:1D.weak} below, we prove that the intuition obtained from this simple self-adjoint argument is indeed correct.

The free resolvent $(H-z)^{-1}$, $z \in \rho(H)$, is an integral operator with the kernel (see \cite{Cuenin-2014-15} for details)
\begin{equation}\label{H.res.ker}
\begin{aligned}
\cR(x,y;z) &= \cN(x,y;z) e^{\ii k(z)|x-y|}, \quad 
\\
\cN(x,y;z) &:=
\frac \ii 2  
\begin{pmatrix}
\zeta(z) & \sgn (x-y)
\\
\sgn (x-y) & \zeta(z)^{-1}
\end{pmatrix}
\end{aligned}
\end{equation}
where
\begin{equation}\label{zeta.k.def}
\zeta(z) := \frac{z+m}{k(z)}, \quad k(z) := (z^2 - m^2)^\frac12
\end{equation}
and the square root on $\C \setminus [0,\infty)$ is chosen such that $\Im k(z)>0$. 

A natural technique is the Birman-Schwinger principle, derived for the Dirac operator e.g.\ in \cite[Thm.~6.1]{Cuenin-2014-15}. 
The Birman-Schwinger operator $Q(z)$ is an integral operator with the kernel
\begin{equation}
\cQ(x,y;z) = \cA(x) \cN(x,y;z) e^{\ii k(z)|x-y|} \cB(y)
\end{equation}
where the factorization of $V$ is based on its polar decomposition $V=U_V |V|$, namely
\begin{equation}
V = \cB \cA, \quad \cB:=U_V |V|^\frac 12, \quad  \cA:=|V|^\frac 12.
\end{equation}
Notice that 
\begin{equation}\label{AB.norm}
\|\cA(x)\|^2 = \|\cB(x)\|^2 = \|V(x)\|.
\end{equation}

As usual, we split $Q=Q(z)$ into a singular and a regular part $L$ and $M$, respectively, i.e.~
\begin{equation}\label{Q.def}
Q=L+M;
\end{equation}
the corresponding ($z$-dependent) kernels read 
\begin{align}\label{L.def}
\cL(x,y) &= \cA(x) 
\Upsilon
\cB(y),
\\
\cM(x,y) &= \cA(x) 
\Big(
\sgn(x-y) \sigma_1
e^{\ii k(z)|x-y|} 
+\Upsilon
(e^{\ii k(z)|x-y|}-1)
\Big)
\cB(y),
\label{M.def}
\end{align}
where
\begin{equation}\label{Ups.def}
\Upsilon = \Upsilon(\zeta(z))=
\frac{\ii }{2}
\begin{pmatrix}
\zeta(z) & 0
\\
0 & \zeta(z)^{-1}
\end{pmatrix}.
\end{equation}
Similarly as in \cite[Sec.~2]{Cuenin-2014-15}, estimating the quadratic form of $L$, we obtain the bound 
\begin{equation}\label{L.norm}
\|L\| \leq \|\Upsilon\|  \|V\|_1 = \frac12 \max\{|\zeta(z)|,|\zeta(z)|^{-1}\} \|V\|_1.
\end{equation}
For later use, we also notice that
\begin{equation}\label{Full Birman-Schwinger HS norm}
\|Q\|_{\mathfrak S ^2}^2 \leq  \left(\frac12+\|\Upsilon\|_{\mathfrak S ^2}^2\right)  \|V\|_1^2
= \frac14 \left(2+|\zeta(z)|^2+|\zeta(z)|^{-2}\right)  \|V\|_1^2.
\end{equation}

The following lemma shows that the possible singularities for $z = \pm m$ of the regular part $M$ are weaker than those of $L$.

\begin{lemma}\label{lem:M.HS}
Let $V$ be as in \eqref{V.def}, $M$ be the integral operator with the kernel \eqref{M.def} and  $\Upsilon$ as in~\eqref{Ups.def}. 
Then 
\begin{equation}\label{M.HS}
\|M\| =  o \big( \|\Upsilon\| \big),  \quad z \to \pm m, \ z \notin \sigma(H).
\end{equation}
\end{lemma}
\begin{proof}
The proof is inspired by \cite[Lem.~1]{Klaus-1977-108}. Since $\Im k(z)>0$ for $z \notin \sigma(H)$, straightforward estimates and \eqref{AB.norm} show that there exists $C>0$ such that
\begin{equation}
\|\Upsilon\|^{-2}\|\cM(x,y)\|^2 \leq C \|V(x)\| \|V(y)\|
\left(
\|\Upsilon\|^{-2} + 1
\right), \quad z \notin \sigma(H).
\end{equation}
Thus, for $\|\Upsilon\|>1$, the function $(x,y) \mapsto \|\Upsilon\|^{-2}\|\cM(x,y)\|^2$ has an integrable upper bound. Since $\|\Upsilon\| \to \infty$ when $k(z) \to 0$ and
\begin{equation}
\lim_{k(z) \to 0}(e^{\ii k(z)|x-y|}-1) = 0, 
\end{equation}
the dominated convergence theorem yields
\begin{align*}
& \lim_{k(z) \to 0} \|\Upsilon\|^{-2} \int_{\R^2}  \|\cM(x,y)\|^2 \; \dd x \, \dd y=  0.
\qedhere
\end{align*}
%
\end{proof}

\begin{theorem}\label{thm:1D.weak}
Let $H$ be as in \eqref{H.def}, $V$ as in \eqref{V.def} and let $m>0$. Define the matrix
\begin{equation}\label{U0.def}
U:=\int_{\R} V(x) \, \dd x.
\end{equation}
If $\Re U_{11} <0$, then, for all sufficiently small $\eps >0$, there exists an eigenvalue $z_+(\eps)$ of $H + \eps V$ satisfying
\begin{equation}\label{z+.exp}
z_+(\eps) = m - \frac{m}{2} U_{11}^2 \eps^2 + o(\eps^2), \quad \eps \to 0+.
\end{equation}
Similarly, if $\Re U_{22} >0$
then, for all sufficiently small $\eps >0$, there exists an eigenvalue $z_-(\eps)$ of $H + \eps V$ satisfying
\begin{equation}\label{z-.exp}
z_-(\eps) = -m + \frac{m}{2} U_{22}^2 \eps^2 + o(\eps^2), \quad \eps \to 0+.
\end{equation}
\end{theorem}
\begin{proof}
According to the Birman-Schwinger principle, see~\cite[Thm.~6.1]{Cuenin-2014-15}, $z \in \spp(H + \eps V)$ if and only if $-1 \in \spp(\eps Q(z))$, where $Q$ is as in~\eqref{Q.def}. Thus we investigate the invertibility of $1 + \eps Q(z)$. Notice that, if $\eps \|M\|<1$, we have 
\begin{equation}\label{Q.decom}
I + \eps Q(z) = (I + \eps M)(I + (I + \eps M)^{-1}\eps L),
\end{equation}
so the invertibility of $Q(z)$ depends on the invertibility of the second factor in~\eqref{Q.decom}. We proceed with the analysis of the latter, find for which $z$ it is not invertible and show that, for these $z$, the condition $\eps \|M\|<1$ holds.

The crucial observation is that the kernel of $L$ is separated (in $x$ and $y$), hence $(I + \eps M)^{-1}\eps L$ is of finite rank and so $-1 \in \spp((I + \eps M)^{-1}\eps L)$ if and only if
\begin{equation}\label{detU.eq}
\det(I_2 +(I + \eps M)^{-1}\eps L) =0.
\end{equation}
To analyze \eqref{detU.eq}, it is convenient to write
\begin{equation}\label{L.AB.hat}
(I + \eps M)^{-1} L = \widehat A_\eps \Upsilon \widehat B,
\end{equation}
where 
\begin{equation}\label{AB.hat.det}
\begin{aligned}
\widehat B&: \quad L^2(\R;\C^2) \to \C^2: \quad  \Psi \mapsto \int_{\R} \cB(x)\Psi(x) \, \dd x,
\\
\widehat A_\eps&: \quad \C^2 \to L^2(\R;\C^2): \quad \Phi \mapsto (I+\eps M)^{-1} \cA(x) \Phi=: \cA_\eps(x) \Phi,
\end{aligned}
\end{equation}
where $\cA_\eps(x)$ is a matrix satisfying
\begin{equation}\label{Aeps.est}
\|\cA_\eps(x)-\cA(x)\| \leq r \|\cA(x)\|, \qquad r=r(\eps,z)=\frac{\eps \|M\|}{1-\eps \|M\|}.
\end{equation}
Notice also that (with $U$ as in \eqref{U0.def})
\begin{equation}\label{U1.est}
\widehat{B} \widehat{A}_\eps = U + U_1, \quad \|U_1\| \leq r \pd{\|V\|_1}.
\end{equation}
Employing these observations and Sylvester's determinant identity, we can rewrite \eqref{detU.eq} as
\begin{equation}
0 = \det(I_2 + \eps\Upsilon (U + U_1)) = 1 + \eps \Tr (\Upsilon (U + U_1))  -\frac14 \eps^2 \det (U + U_1).
\end{equation}
Since
\begin{equation}\label{TrU0}
\Tr (U\Upsilon) =  \frac{\ii}{2} (\zeta U_{11} + \zeta^{-1} U_{22}),
\end{equation}
our initial guess (ignoring the smaller terms) for the dependence of $\zeta$ on $\eps$ reads 
\begin{equation}
\zeta^0_+= \frac{2 \ii}{\eps U_{11}}, \qquad \zeta^0_-= - \frac{\ii}{2} \eps U_{22}.
\end{equation}
Notice that the assumption that $\Re U_{11}<0$ or $\Re U_{22} > 0$ is needed since the allowed region in terms of $\zeta$ is $\Im \zeta <0$, see the discussion of \cite{Cuenin-2014-79} after (2.6) there.

Finally, we prove that there is indeed a solution $\zeta_+$ of~\eqref{detU.eq} in a neighborhood of $\zeta^0_+$; the reasoning for $\zeta_-$ is analogous. To this end, for $|\alpha|<\delta_0$, we define $\zeta_\alpha:=\zeta^0_+(1+\alpha)$ and select $\delta_0$ so small that, with some $\beta>0$, we have $\Im \zeta_\alpha < - \beta <0$ for all $|\alpha|<\delta_0$ and $\eps \to 0+$. Observe that with this $\zeta_\alpha$, we have (uniformly in $\alpha$) that  $\eps\|M\| = o(\eps |\zeta_+^0|) = o(1)$ as $\eps \to 0+$ and so $r = o(1)$ as $\eps \to 0+$; 
the former justifies \eqref{Q.decom} in particular.

Take $\zeta=\zeta_\alpha$ and define the function 
\begin{equation}
F(\alpha):=
\det (
I_2 + \Upsilon (U+U_1)
), \quad |\alpha|<\delta_0.
\end{equation}
Since, for all $|\alpha|<\delta_0$, we have $\Im \zeta_\alpha < - \beta < 0$ and $\zeta_\alpha \to \infty$ as $\eps \to 0+$, i.e.~the corresponding $z_\alpha \in \rho(H)$ and $z_\alpha \to m$ as $\eps \to 0+$, the function $F$ is holomorphic for $|\alpha|<\delta_0$. Moreover, using \eqref{Aeps.est}, \eqref{U1.est}, we get 
\begin{equation}
\begin{aligned}
F(\alpha) &= 1 + \eps \Tr (U \Upsilon ) + \eps \Tr (U_1 \Upsilon ) -\frac14 \eps^2  \det (U+U_1) 
\\
& = 1 + \big(-(1+\alpha) + \BigO(\eps^2) \big) + o(1) + \eps^2 \BigO(1)
\\
& = -\alpha + o(1), \qquad \eps \to 0+.
\end{aligned}
\end{equation}
Hence, Rouch\'e's theorem implies that, for all $0<\eps<\eps_{\delta_0}$, functions $F(\alpha)$ and $G(\alpha):=-\alpha$ have the same number of zeros in the ball $B_{\delta_0}(0)$. Notice that the same reasoning is valid for any $0<\delta<\delta_0$, thus we obtain that the sought solution of~\eqref{detU.eq} reads 
\begin{equation}\label{zetap.sol}
\zeta_+ = \zeta^0_+(1+o(1)), \quad \eps \to 0+. 
\end{equation}
The last step, yielding~\eqref{z+.exp}, is to use the relation \eqref{zeta.k.def} between $\zeta$ and $z$ and rewrite \eqref{zetap.sol} in terms of $z_+$. 
\end{proof}

\begin{remark}\label{rem:wc}
\begin{enumerate}[\upshape i)]
\item Theorem \ref{thm:1D.weak} shows that the spectral estimate \eqref{spectral inclusion} is sharp in the weak coupling regime. Indeed, the latter can be stated as
\begin{equation}
|z\mp m|\leq \frac{m}{2}\eps^2\|V\|_1^2+o(\eps^2), \quad \eps \to 0+.
\end{equation}
\item Notice that in the proof of Theorem~\ref{thm:1D.weak}, we use that $\|M\| = o(\|\Upsilon(\zeta(z))\|)$ as $z \to \pm m$, $z \notin \sigma(H)$ only and not the particular structure of $M$. The latter would be needed to derive more terms in the expansions \eqref{z+.exp}, \eqref{z-.exp}. 	
\item It is known that the weak coupling limit for the Dirac operator is equivalent to the non-relativistic limit. We do not pursue this connection here and refer to \cite{Cuenin-2014-79} for a discussion in this matter. 
\end{enumerate}
\end{remark}

\subsection{Lieb-Thirring inequalities}
\label{subsec:1D.LT}

In the last subsection we have seen that the massive Dirac operator is critical, \ie~an arbitrarily small (non-self-adjoint) perturbation will create an eigenvalue. Here we prove an upper bound for the number of eigenvalues in certain subsets of the complex plane. The upper bound will be a consequence of a Lieb-Thirring type inequality. We prove similar results for the (non-critical) massless Dirac operator.

\begin{theorem}\label{thm.LT}
Let $H$ be as in \eqref{H.def} and $V$ as in \eqref{V.def}. If $m=0$ and $\|V\|_1\geq 1$, then 
\begin{equation}\label{LT meq0}
\sum_{z\in\spd(H+V)}\frac{\dist(z,\sigma(H))}{(|z|+1)^{2}}\leq C(1+\|V\|_2^4)\|V\|_1^2,
\end{equation}
where each eigenvalue is counted according to its algebraic multiplicity.
If $m>0$, then for any $\tau>0$ we have that
\begin{equation}\label{LT mg0}
\sum_{z\in\spd(H+V)}\frac{\dist(z,\sigma(H))|m^2-z^2|^{\frac{\tau}{2}}}{(m+|z|)^{2+\tau}}\leq C_{\tau} \frac{A_m(V)}{m} \max \{\|V\|_1,\|V\|_1^2 \}
\end{equation}
where
\begin{equation}
A_m(V)=\begin{cases}
\ds
\min \left\{
\frac{1}{1-\|V\|_1 e^{\frac{1}{2}(\|V\|_1+1)^2}}, \frac{\left(1+m^{-2}\|V\|_2^4\right)^{2+\tau}}{\rho_0^2} \right\} & \text{ if  }  \|V\|_1< \rho_0,
\\[4mm]
\ds
\frac{\left(1+m^{-2}\|V\|_2^4\right)^{2+\tau}}{\rho_0^2} & \text{ if  }\|V\|_1\geq \rho_0,
\end{cases}
\end{equation}
and where $\rho_0$ is the unique solution to $\rho e^{\frac{1}{2}(\rho+1)^2}=1$ 
($\rho_0\approx 0.38$).
\end{theorem}

\begin{remark}
\begin{enumerate}[\upshape i)]
\item We recall that in the massless case ($m=0$) the spectral inclusion \eqref{spectral inclusion} states that there are no eigenvalues whenever $\|V\|_1<1$. For this reason we are assuming that $\|V\|_1\geq 1$ above.
\item We emphasize the dependence of the bound on the $L^1$ norm of $V$. One reason is that this norm is invariant with respect to rescaling of the mass, i.e.\ the substitution $V\to m^{-1}V(m^{-1}\cdot)$ does not change the $L^1$ norm. Secondly, a straightforward adaptation of the proof shows that if instead of $\|V\|\in L^2(\R)$ we assume that $\|V\|\in L^p(\R)$ for some $p>1$, then \eqref{LT meq0} and \eqref{LT mg0} hold with $\|V\|_2^2$ replaced by $\|V\|_p^{2p/(p-1)}$.
\item The bounds of Theorem \ref{thm.LT} should be compared with those of \cite{Dubuisson-2014-78}. In the one-dimensional case it is claimed that, for $p>1$ and e.g.\ for $m=0$, one has
\begin{equation}\label{Dubuisson massless}
\sum_{z\in\spd(H+V)}\frac{\dist(z,\sigma(H))^{p+\tau}}{(|z|+1)^{2(p+\tau)}}\leq C_{p,\tau}\|V\|_p^p
\end{equation}
for any $\tau\in (0,\min(p-1,1))$. In fact, an inspection of the proof shows that the constant $C_{p,\tau}$ still depends on the potential through the parameter $b$ introduced in Section 4.1 there. Our estimate \eqref{LT mg0} yields better weights both for sequences of eigenvalues accumulating to a point in the essential spectrum or tending to infinity. Moreover, the constant is universal; the dependence on (the $L^2$ norm of) the potential is exhibited explicitly. The comparison in the massive case is less obvious, and we will not pursue the issue here. 
\item[iv)] In the self-adjoint case, Lieb-Thirring type inequalities for one-dimensional Dirac operators may be found in \cite{Frank-2011-157}. Note that even there the eigenvalue sums cannot be controlled by a single $L^p$ norm. Similar inequalities for resonances were established in \cite{Iantchenko-2014-256} and \cite{Korotyaev-2014-104} in the massless case and in \cite{Iantchenko-2014-420} for the massive case. 
\end{enumerate}
\end{remark} 
Theorem \ref{thm.LT} has the following consequences for the number of eigenvalues $N_K (H+V)$ of $H+V$ in a compact subset $K\subset\rho(H)$. For any $\delta,R> 0$ and $\epsilon\geq 0$ we set 
\begin{equation}
K_{\delta,\epsilon,R}=\{z\in\C:\dist(z,\sigma(H))\geq \delta,\,\dist(z,\{-m,+m\})\geq \epsilon,\,|z|\leq R\}.
\end{equation}
\begin{corollary}
If $m=0$ and $\|V\|_1\geq 1$, then we have
\begin{equation}
N_{K_{\delta,0,R}}\leq \frac{C}{\delta}(1+R^2)(1+\|V\|_2^4)\|V\|_1^2.
\end{equation}
If $m>0$, then we have
\begin{equation}\label{Number of eigenvalues mg0}
N_{K_{\delta,\epsilon,R}}\leq \frac{C_{\tau}}{\delta} \max\{ m^{1+\frac \tau 2}\epsilon^{-\frac \tau 2},R^{2}\} \frac{A_{m}(V)}{m}
\max\{\|V\|_1,\|V\|_1^2\}.
\end{equation}
\end{corollary}

\begin{proof}
For $m>0$ the claim follows from the lower bound
\begin{equation}
\frac{\dist(z,\sigma(H))|m^2-z^2|^{\frac{\tau}{2}}}{(m+|z|)^{2+\tau}}
\geq c_{\tau}\delta\min \{ m^{-1-\frac\tau 2}\epsilon^{\frac \tau 2},R^{-2} \}.
\end{equation}
This can be seen by treating the cases $|z|\leq 2m$ and $|z|>2m$ separately and using that $|z^2-m^2|\geq m\epsilon$ in the first case. The case $m=0$ is even easier.
\end{proof}

\begin{proof}[Proof of Theorem \ref{thm.LT}]
The proof is based on complex analysis. The general approach is explained very well in \cite{Demuth-2013-232} and we refer the reader to this article for more details. We treat the massive and the massless case separately, starting with the latter.
For simplicity we prove \eqref{LT meq0} only for eigenvalues in the upper half plane~$\C^+$; the proof for the lower half plane is analogous. The basic idea in the complex analysis approach is to relate the eigenvalues to the zeros of a holomorphic function. It is convenient to define the following maps (recall that $Q=|V|^\frac 12 R(z)V^\frac 12$):
\begin{equation}\label{conformal map massless}
\begin{split}
&h:\C^+\to \C,\quad h(z):={\det}_2(I+Q),\\
&\varphi:\C^+\to\mathbb{D},\quad \varphi(z):=\frac{z-\ii}{z+\ii},\\
&\nu:\mathbb{D}\to\mathbb{D},\quad \nu(w):=\frac{w+\varphi(\ii\eta)}{1+\varphi(\ii\eta)w},\\
&\psi:\C^+\to\mathbb{D},\quad\psi(z):=\nu^{-1}(\varphi(z)),\\
&g:\mathbb{D}\to\C,\quad g(w):=\frac{h(\psi^{-1}(w))}{h(\ii\eta)},
\end{split}
\end{equation}
where $\ii\eta\in \rho(H+V)$ will be chosen momentarily. Note that $\varphi,\nu,\psi$ are conformal maps and the regularized determinant $\det_2$ is defined for any $T\in\mathfrak{S}^2$ by
\begin{equation}
{\det}_2(I+T):=\det\left((I+T) e^{-T}\right).
\end{equation}
We then have that $\det_2(I+T)=0$ if and only if $(I+T)$ is not invertible; see \eg~\cite[Thm.~9.2]{Simon-2005}. In particular, $h(z)=0$ if and only if $z\in \spp(H+V)$. Moreover, $h$ is analytic in the Hilbert-Schmidt norm, see e.g.\ \cite{Frank-2014}.
From \cite[Thm.~9.2]{Simon-2005} and \eqref{Full Birman-Schwinger HS norm} we have the uniform bound (notice the $\|\Upsilon\|_{\mathfrak{S}^2}^2=1/2$ for $m=0$)
\begin{equation}\label{uniform bound for log h massless}
\log|h(z)|\leq \Gamma_2\|Q(z
)\|_{\mathfrak{S}^2}^2
\leq \frac{1}{2} \|V\|_1^2.
\end{equation}
The value of the optimal constant is $\Gamma_2=1/2$, see e.g.\ formula (2.2) in Chaper IV of \cite{Gohberg-1969}. 
This implies the following estimate,
\begin{equation}
\log|g(w)|\leq \frac{1}{2}\|V\|_1^2-\log|h(\ii\eta)|.
\end{equation}
We will use the $L^2$ norm to estimate the second term. First observe that
\begin{equation}\label{eq. HS bound massless with L2 norm}
\|Q(\ii \eta)\|_{\mathfrak{S}^2}\leq 2\||V|^\frac12 |\sqrt{-\Delta}-\ii\eta|^{-\frac12}\|_{\mathfrak{S}^4}^2
\leq (2\pi)^{-\frac{1}{4}}2\sqrt{\pi}\|V\|_2\eta^{-\frac{1}{2}},
\end{equation} 
where we used the Schwarz inequality in Schatten spaces in the first and the Kato-Seiler-Simon bound \cite[Thm.~4.1]{Simon-2005} in the second estimate; a more precise bound is stated in \cite[Theorem 4.3]{Cuenin-2014-15}. We also used that the norm of the kernel of $|R_0(\ii\eta)|^{1/2}$ is dominated by twice the absolute value of the kernel of $|\sqrt{-\Delta}-\ii\eta|^{-\frac12}$ and that the $L^4$ norm of the function $||\cdot|-\ii\eta|^{-\frac12}$ is bounded from above by $\pi^{\frac 14}\eta^{-\frac14}$. We set 
\begin{equation}\label{eq. choice of eta}
\eta^{\frac12}=\gamma(2\pi)^{-\frac{1}{4}}2\sqrt{\pi}\|V\|_2,\quad \text{where} \quad  \gamma=e^\frac 52.
\end{equation}
This guarantees that $\ii\eta\in\rho(H+V)$ since $\gamma^{-1}<1$. By \cite[Thm.~9.2]{Simon-2005} we have the bound
\begin{equation}
|h(\ii\eta)-1|\leq \gamma^{-1}\exp\left(\frac12(\gamma^{-1}+1)^2\right)\leq \gamma^{-1} e^2\leq e^{-\frac 12}.
\end{equation}
and so, since $\|V\|_1 \geq 1$, 
\begin{equation}\label{eq. bound on g(w)}
\log|g(w)|\leq\frac12 \|V\|_1^2 + \frac 12 \leq \|V\|_1^2.
\end{equation}
Since $g$ is holomorphic on the unit disk $\mathbb{D}$ and continuous up to the boundary, the following Blaschke condition holds,
\begin{equation}\label{eq. Blaschke}
\sum_{g(w)=0}(1-|w|)\leq \sum_{g(w)=0}\left|\frac{1}{w}\right|\leq 
\sup_{0<r<1}\frac{1}{2\pi}\int_0^{2\pi}\log|g(r\e^{\ii \theta})|\dd\theta\leq \log\|g\|_{\infty}.
\end{equation}
Here and in the following, every zero is counted according to its multiplicity. We have also used the normalization condition $g(0)=1$ in \eqref{eq. Blaschke}, which follows from $\psi(\ii\eta)=0$. To translate the result back to the $z$-plane we use the distortion bound 
\begin{equation}\label{eq. Koebe}
(1-|w|)\sim |\psi'(z)|\dist(z,\sigma(H))\geq \frac{2}{1+\eta^2}\frac{\dist(z,\sigma(H))}{|z+\ii|^2}.
\end{equation} 
The first inequailty follows from the Koebe distortion theorem, the second from an explicit computation.
Combination of \eqref{eq. bound on g(w)}, \eqref{eq. Blaschke} and \eqref{eq. Koebe} yields 
\begin{equation}
\sum_{z\in\spp(H+V)}\frac{\dist(z,\sigma(H))}{|z+\ii|^{2}}\leq C(1+\eta^2) \|V\|_1^2.
\end{equation}
By the choice of $\eta$, this is equivalent to \eqref{LT meq0} for $z\in\C^+$.

Next, we consider the massive case. Without loss of generality we may restrict our attention to the case $m=1$; the general case follows by scaling. We use the same maps as in \eqref{conformal map massless} except that we replace $\varphi$ by
\begin{equation}
\varphi:\rho(H)\to\mathbb{D},\quad \varphi(z):=\frac{\sqrt{(z+1)/(z-1)}-\ii}{\sqrt{(z+1)/(z-1)}+\ii}.
\end{equation}
Instead of \eqref{uniform bound for log h massless} we now have (again from \eqref{Full Birman-Schwinger HS norm}) the estimate
\begin{equation}
\log|h(z)| 
\leq \frac12 \|Q(z)\|_{\mathfrak{S}^2}^2
= \frac 18 \left(2 + |\zeta(z)|^2 + |\zeta(z)|^{-2}\right) \|V\|_1^2.
\end{equation}
Assume first that $\|V\|_1<\rho_0$. By \eqref{spectral inclusion} we may choose $\eta>0$ (independent of $V$) such that $\ii\eta\in\rho(H+V)$, \ie~$h(\ii \eta)\neq 0$. Since $|\zeta(\ii \eta)|=1$, we have
\begin{equation}
\|Q(\ii \eta)\|_{\mathfrak{S}^2} \leq \|V\|_1.
\end{equation}
It follows that
\begin{equation}
|h(\ii\eta)-1| \leq \|V\|_1 \exp\left(\frac{1}{2}(\|V\|_1+1)^2\right),
\end{equation}
and thus 
\begin{equation}\label{eq. bound on g(w) massive}
\begin{aligned}
\log|g(w)|&
\leq \frac 18 \left(2 + |\zeta(\psi^{-1}(w))|^2 + |\zeta(\psi^{-1}(w))|^{-2}\right) \|V\|_1^2
+\frac{e^{\frac{1}{2}(\|V\|_1+1)^2}\|V\|_1}{1-\|V\|_1 e^{\frac{1}{2}(\|V\|_1+1)^2}}
\\
&\leq C\|V\|_1(1-\|V\|_1 e^{\frac{1}{2}(\|V\|_1+1)^2})^{-1}|1-w|^{-2}|1+w|^{-2}.
\end{aligned}
\end{equation}
In the first estimate we used the triangle inequality $|h(\ii\eta)|\geq 1-|h(\ii\eta)-1|$, the monotonicity of the logarithm and the elementary inequality $\log(1-x)\geq -x/(1-x)$ for $x=|h(\ii\eta)-1|\in (0,1)$.
The second estimate follows from the definition of $\psi$, some elementary inequalities for positive numbers and the fact that $\|V\|_1^2<\|V\|_1$.
Since the right hand side is no longer bounded, we cannot use \eqref{eq. Blaschke}. Instead, we have to use a more refined result due to \cite{Borichev-2009-41}. In our case it implies that
\begin{equation}\label{eq. sum zeros of g(w) massive}
\sum_{w:\,g(w)=0}(1-|w|)|1-w|^{1+\tau}|1+w|^{1+\tau}\leq C_{\tau} (1-\|V\|_1 e^{\frac{1}{2}(\|V\|_1+1)^2})^{-1}\|V\|_1
\end{equation}
for any $\tau>0$. Using the estimates
\begin{equation}\label{eq. distortion bounds}
\begin{split}
(1-|w|)&\gtrsim\dist(z,\sigma(H))(1+|z|)^{-1}|z^2-1|^{-\frac{1}{2}}(1+\eta^2)^{-1},\\
|1-w||1+w|&\gtrsim (1+|z|)^{-1}|z^2-1|^{\frac{1}{2}}(1+\eta^2)^{-1},
\end{split}
\end{equation}
which follow from straightforward albeit tedious computations, directly from the definition of $\psi$ and by Koebe's distortion theorem we infer from \eqref{eq. sum zeros of g(w) massive} that
\begin{equation}\label{eq. half of LT mg0}
\sum_{z\in\spd(H+V)}\frac{\dist(z,\sigma(H))|1-z^2|^{\frac{\tau}{2}}}{(1+|z|)^{2+\tau}}\leq C_{\tau} (1-\|V\|_1 e^{\frac{1}{2}(\|V\|_1+1)^2})^{-1}\|V\|_1.
\end{equation}
This is half of the desired bound \eqref{LT mg0} for $m=1$ and $\|V\|_1<\rho_0$.
Observe that if $\|V\|_1<\rho_0/2$, then the right hand side of \eqref{eq. half of LT mg0} is bounded from above by a constant multiple of $\max\{\|V\|_1,\|V\|_1^2\}$, while $A_m(V)$ is bounded from below by a positive constant. Hence, it remains to prove \eqref{LT mg0} for $\|V\|_1\geq \rho_0/2$ and with $A_m(V)=(1+\|V\|_2^4)^{2+\tau}$.
Since we have $|\sqrt{1-\Delta}-\ii\eta|\geq |\sqrt{-\Delta}-\ii\eta|$, \eqref{eq. HS bound massless with L2 norm} holds, and we make the same choice of $\eta$ and $\gamma$ as in the massless case. A repetition of the above arguments yields that 
\begin{equation}
\log|g(w)|\leq 4\frac{\|V\|_1^2+\gamma^{-1}}{|1-w|^2|1+w|^2}
\end{equation}
and finally
\begin{equation}
\sum_{z\in\spd(H+V)}\frac{\dist(z,\sigma(H))|1-z^2|^{\frac{\tau}{2}}}{(1+|z|)^{2+\tau}}\leq C_{\tau}(1+\eta^2)^{2+\tau}(\|V\|_1^2+\gamma^{-1}).
\end{equation}
For $\|V\|_1\geq \rho_0/2$, we have that $\gamma^{-1}\leq 4(\gamma\rho_0^2)^{-1}\|V\|_1^2$. Hence, by the choice of $\eta$ and $\gamma$, this is the desired bound.
\end{proof}

\section{Applications}
\label{sec:appl}

\subsection{Damped wave equation in $L^2(\R)$}
\label{subsec:DWE}

Our firs non-self-adjoint application motivated from physics is the damped wave equation
\begin{equation}\label{DWE.int}
u_{tt}(t,x) + 2 a(x) u_t(t,x) = u_{xx}(t,x) -q_0 u(t,x), \quad t >0, \ x \in \R;
\end{equation} 
where the damping $a$ and the potential $q_0 \in \R_+$ satisfy 
\begin{equation}\label{a.q.def}
a(x) = a_0 + a_1(x), \quad a_0>0, \ a_1 \in L^1(\R) \cap L^2(\R), \quad q_0>a_0^2.
\end{equation}

The second order scalar equation \eqref{DWE.int} can be reformulated as a first order system, suitable for spectral analysis, in the form
\begin{equation}\label{DWE.G}
\partial_t 
\begin{pmatrix}
u \\ v
\end{pmatrix}
=
\underbrace{\begin{pmatrix}
-2 a & - \partial_x - q_0^\frac12 \\
- \partial_x + q_0^\frac12 & 0
\end{pmatrix}
}_{G}
\begin{pmatrix}
u \\ v
\end{pmatrix}.
\end{equation}
The equivalence to another (perhaps more intuitive) system with $v=u_t$ is extensively discussed \eg~in \cite{Gesztesy-2011-251}. We view $G$ as an operator in $L^2(\R;\C^2)$ and we employ a similarity transformation that brings $G$ to the form studied in Section~\ref{sec:1D}. Let
\begin{equation}\label{T.m.def}
T= - \frac{1}{2(\mu^2+\ii a_0 \mu)^\frac12}
\begin{pmatrix}
q_0^\frac12 & a_0 - \ii \mu 
\\
a_0 - \ii \mu & q_0^\frac12
\end{pmatrix},
\qquad \mu:=(q_0-a_0^2)^\frac12>0.
\end{equation}
Then a straightforward calculation yields
\begin{equation}\label{G.transf}
T \ii G T^{-1} = 
\underbrace{
\begin{pmatrix}
\mu & - \ii \partial_x
\\
- \ii \partial_x & -\mu 
\end{pmatrix}
}_{H}
+ 
\underbrace{
\frac{a_1(x)}{\mu} 
\begin{pmatrix}
-a_0 & q_0^\frac12
\\
- q_0^\frac12 & a_0
\end{pmatrix}
}_{V}
- \ii a_0 I_2.
\end{equation}
In the simplest case with $a_1=0$, \eqref{G.transf} immediately gives that
\begin{equation}
\sigma(G) = \se{3}(G) =  -a_0 + (-\ii \infty, -\ii \mu] \cup [\ii \mu, + \ii\infty),
\end{equation}
and thus illustrates the well-known spectral picture exhibiting the effect of damping (the shift of the spectrum to the left). As a corollary of the claims established or recalled in Section~\ref{sec:1D}, we obtain new results in the situation when damping is no longer constant but satisfies \eqref{a.q.def}. 

\begin{theorem}\label{thm:DWE}
Let $G$, $H$, $V$, $a$, $q_0$ and $\mu$ be as above. Then 
\begin{equation}
\se{3}(G)=-a_0 + (-\ii \infty, -\ii \mu] \cup [\ii \mu, + \ii\infty)
\end{equation}
and the following holds.
\begin{enumerate}[\upshape i)]
\item For any $\tau>0$, the Lieb-Thirring type inequality \eqref{LT mg0} and the bound on the number of eigenvalues \eqref{Number of eigenvalues mg0} hold with
\begin{equation}\label{V1.a1}
\|V\|_1= \sqrt{\frac{q_0^\frac12 + a_0}{q_0^\frac12 - a_0}} \int_{\R} |a_1(x)| \, \dd x,
\quad 
\|V\|_2= \sqrt{\frac{q_0^\frac12 + a_0}{q_0^\frac12 - a_0}} \left(\int_{\R} |a_1(x)|^2 \, \dd x\right)^{\frac 12}
\end{equation}
and with the replacement $H \mapsto \ii H$ and $m \mapsto \ii\mu$.
\item If $\|V\|_1<1$, then 
\begin{equation}\label{G.ind.b}
\spp(G) \setminus \se{3}(G) \subset \ov B_{\mu r_0}(\ii \mu x_0-a_0) \, \dot{\cup} \, \ov B_{\mu r_0}(-\ii \mu x_0 - a_0), 
\end{equation}
where $x_0$ and $r_0$ are as in \eqref{x0r0} with $\|V\|_1$ as in \eqref{V1.a1}.
\item In the weak coupling regime, \ie~$a_1$ is replaced everywhere by $\eps a_1$ with $\eps>0$: if $\int_{\R} a_1(x) \, \dd x >0$, then, for all sufficiently small $\eps>0$, there are two eigenvalues $z_\pm(\eps)$, $z_-=\ov{z_+}$, of $G$ satisfying
\begin{equation}
z_+(\eps) = -a_0 + \ii \mu - \frac{\ii \pd{a_0^2}}{2 \mu} \left(\int_\R a_1(x) \; \dd x \right)^2 \eps^2 + o(\eps^2), \quad \eps \to 0+.
\end{equation}
\end{enumerate}
\end{theorem}

\subsection{Armchair graphene waveguides}
\label{subsec:ac}

As a second application motivated from physics, we consider the two-dimensional Dirac operator of an infinite straight graphene waveguide $\Omega=(-a,a)\times\R$,
\begin{equation}\label{D0.def}
D=
\begin{pmatrix}
0&\tau^*&0&0\\
\tau&0&0&0\\
0&0&0&-\tau\\
0&0&-\tau^*&0\\
\end{pmatrix}\quad \mbox{in  } L^2(\Omega;\C^4),
\end{equation}
where $\tau:=-\ii \partial_1+\partial_2$ and $\tau^*:=-\ii\partial_1 - \partial_2$ is the formal adjoint. The domain of $D$ consists of spinors $\psi\in H^1(\Omega;\C^4)$ satisfying so-called armchair boundary conditions 
\begin{equation}\label{boundary conditions}
\psi_i(-a,x_2)=\psi_{i+2}(-a,x_2),
\quad 
\psi_i(a,x_2)=e^{\ii\Theta}\psi_{i+2}(a,x_2), 
\quad i=1,2,
\end{equation}
where $0\leq\Theta<2\pi$ depends on the geometry of the waveguide. It was proved in \cite[Prop.~1,~19]{Freitas-2014-26} that $D$ is self-adjoint and
that the spectrum is given by
\begin{equation}
\sigma(D)=\se{3}(D)=(-\infty,-E_0]\cup[E_0,\infty),
\end{equation}
where
\begin{equation}\label{def. E0 mun}
E_0:=\min_{n\in\Z}|\xi_n|, \qquad \xi_n:=\frac{\pi n}{2a}-\frac{\Theta}{4a}.
\end{equation}
To simplify the presentation in the sequel, we restrict ourselves to the case when $\Theta \in (0,\pi)$ and thus $E_0 = -\xi_0>0$; the results can be extended in a straightforward way to the other cases.

Although the algebraic structure of Dirac waveguide operators is more complicated than in the Laplacian (or Schr\"odinger) case, it might be helpful to remark that the numbers $\{\xi_n\}_{n \in \Z}$ play the role of spectral thresholds in the essential spectrum of $D$. The corresponding (normalized in $L^2((-a,a);\C^4)$) transverse eigenfunctions read
\begin{equation}\label{Psi.n.def}
\Psi_n^+(x_1) = \frac{1}{2 a^\frac 12}
\begin{pmatrix}
e^{-\ii\xi_n x_1}\\ 
0 \\
(-1)^ne^{-\ii\frac{\Theta}{2}}e^{\ii\xi_n x_1}
\\
0
\end{pmatrix},
\ \ 
\Psi_n^-(x_1) = \frac{1}{2 a^\frac 12}
\begin{pmatrix}
0\\ 
e^{-\ii\xi_n x_1}\\
0\\
(-1)^n e^{-\ii\frac{\Theta}{2}} e^{\ii\xi_n x_1}
\end{pmatrix}
\end{equation}
and the set $\{\Psi_n^\sigma\}_{n \in \Z, \sigma\in\{+,-\}}$ forms an orthonormal basis in $L^2((-a,a);\C^4)$. To proceed with spectral analysis of perturbations of $D$, we derive a convenient representation of the resolvent of $D$ based on its decomposition into transverse modes. Moreover, we employ a unitary transformation bringing $D$ and its resolvent close to the form of the Dirac operator $H$ investigated in Section~\ref{sec:1D}. Notice that $-\xi_n$ plays the role of $m$ in previous formulas.

\begin{lemma}\label{lem:ac.res}
Let $D$ be as in \eqref{D0.def}. Then, for all $z \notin (-\infty,-E_0] \cup [E_0,\infty)$,
\begin{equation}\label{D0.res}
\Sigma (D-z)^{-1} \Sigma^{-1} 
=
\sum_{n \in \Z}
(H_n-z)^{-1}
P_n
=
\sum_{\substack{n \in \Z \\ \sigma \in \{+,-\}}}
(H_n-z)^{-1}
P_n^\sigma
\end{equation}
where
\begin{equation}\label{Sigma.def}
\Sigma = \frac 12
I_2 \otimes 
\begin{pmatrix}
1+\ii & 1+\ii  \\
-1+\ii & 1-\ii
\end{pmatrix},
\quad 
H_n = 
I_2 \otimes 
\begin{pmatrix}
-\xi_n  & - \ii \partial_2 \\ 
- \ii \partial_2 & \xi_n 
\end{pmatrix},
\quad
P_n = P_n^+ + P_n^-
\end{equation}
and, for every $f \in L^2(\Omega;\C^4)$,
\begin{equation}
(P_n^\sigma f)(x_1,x_2) = \langle \Sigma \Psi_n^\sigma, f(\cdot,x_2) \rangle_{L^2((-a,a);\C^4)} \Sigma \Psi_n^\sigma(x_1), \quad \sigma \in \{+,-\}.
\end{equation}
\end{lemma}
\begin{proof}
Notice that, for any $g \in H^1(\R;\C)$, we have
\begin{equation}
\begin{aligned}
D  g(x_2) \Psi_n^\sigma(x_1) 
&=
I_2 \otimes 
\begin{pmatrix}
0 & - \partial_2 - \xi_n  \\ 
\partial_2 - \xi_n & 0 
\end{pmatrix}
g(x_2) \Psi_n^\sigma(x_1)
\\
&= 
\Sigma^{-1} H_n \Sigma
g(x_2) \Psi_n^\sigma(x_1).
\end{aligned}
\end{equation}
Thus \eqref{D0.res} follows by standard arguments.
\end{proof}

In the following, we investigate the spectrum of $D+V$ where
\begin{equation}\label{V.def.ac}
V: \Omega \to \C^{4 \times 4}: \quad \|V\| \in L^1(\Omega)\cap L^2(\Omega).
\end{equation}
To keep the connection to Section \ref{sec:1D}, in the proofs we always employ the unitary transformation $\Sigma$ and thus instead of $V$ we in fact work with
\begin{equation}\label{W.def}
W= \Sigma V \Sigma^{-1}.
\end{equation} 
\begin{remark}[On the assumption $\|V\| \in L^2(\Omega)$]\label{rem:ac.L2}
Similarly to the case of the one-dimensional Dirac operator, the condition $\|V\| \in L^2(\Omega)$ is imposed for convenience only. However, the situation is slightly different for the waveguide: We cannot drop the $L^2$ norm completely, but merely replace it (at the expense of using a more complicated definition of the sum $D+V$) by an $L^{1+\varepsilon}$ norm, where $\varepsilon>0$ is arbitrary. The $\varepsilon$ loss takes place in an orthogonality argument for estimating an infinite sum in the proof of Lemma \ref{lem:M12.est}. We do not know if this is just a technical issue.
\end{remark}

\begin{lemma}\label{lem.e3.ac}
Let $D$ be as in \eqref{D0.def} and $V$ as in \eqref{V.def.ac}. Then
\begin{equation}
\se{3}(D+V) = \se{3}(D) = (-\infty,-E_0]\cup[E_0,\infty).
\end{equation}
\end{lemma}
\begin{proof}
We show $V$ is relatively compact with respect to $D$, so the essential spectrum $\se{3}$ is preserved, see \cite[Thm.~IX.2.1]{EE}. In the following, we employ the unitary transformation $\Sigma$.

First, using the explicit kernel \eqref{H.res.ker} with \pd{$z=0$ and $m=-\xi_n$ and $\|W\| \in L^2(\Omega)$, we  verify that $W H_n^{-1}P_n$ is a Hilbert-Schmidt operator with 
\begin{equation}
\|W  H_n^{-1} P_n\|^2_{\mathfrak{S}^2}\leq C(1+|n|)^{-1}\|W\|_2^2. 
\end{equation}	
}
We will show that the series 
\begin{equation}\label{VHn.ser}
\sum_{n \in \Z} W H_n^{-1} P_n,
\end{equation}
having Hilbert-Schmidt operators as summands, is convergent in the operator norm; this will imply that $V D^{-1}$ is compact.

To show the convergence of \eqref{VHn.ser}, we approximate $W$ by bounded potentials $W_n$ defined by
\begin{equation}
W_n(x):=\begin{cases}
W(x)\quad &\mbox{if  }\|W(x)\|\leq C_n,\\
0\quad &\mbox{if  }\|W(x)\|> C_n,
\end{cases}
\end{equation}
where $C_n$ is chosen such that $\|W\mathbf{1}_{\{x:\|W(x)\|> C_n\}}\|_2\leq 2^{-n}$; this is possible e.g.\ by Chebychev's inequality. In summary, we have chosen $W_n$ such that
\begin{equation}
\|W-W_n\|_2\leq 2^{-n},\quad \|W_n\|_2\leq \|W\|_2.
\end{equation} 
By the mutual orthogonality of $\Ran P_n $, we have for any $N\in\N$ with $N\geq 1$ that
\begin{equation}
\Big\|\sum_{|n|>N}
W_n H_n^{-1}P_n  \Big \|
=
\Big \| \sum_{|n|>N}
P_n H_n^{-1} W_n^* \Big \| 
\leq C\pd{N^{-\frac 12}}\|W\|_2.
\end{equation}
On the other hand, we have
\begin{equation}
\Big \|\sum_{|n|>N}
(W-W_n) H_n^{-1}P_n \Big \| 
\leq C\sum_{|n|>N}2^{-n}.
\end{equation}
The claim is proved.
\end{proof}

\subsubsection{Weakly coupled eigenvalues in armchair waveguides}

We analyze the eigenvalues of $D+ \eps V$ emerging from the edges of the essential spectrum and their asymptotics as $\eps \to 0+$. Here $V$ is assumed to satisfy \eqref{V.def.ac} and thus $\se{3}(D+\eps V) =\se{3}(D)$ for every $\eps \geq 0$ by Lemma~\ref{lem:ac.res}. As in the one-dimensional case, the main tool is the Birman-Schwinger principle. To be able to use the formulas from the one-dimensional case, we employ the unitary transformation $\Sigma$, see \eqref{Sigma.def}, and representation of the resolvent of $D$ in \eqref{D0.res}.

By standard arguments, it can be verified that $z \in \spp(D+ \eps V)$ if and only if $-1 \in \spp( \eps Q)$, where the Birman-Schwinger operator $Q$ has the form
\begin{equation}
Q(z) = \ov{A \Sigma (D-z)^{-1} \Sigma^{-1} B},
\end{equation}
(the bar denotes the closure) and where $A$, $B$ are multiplication operators by the matrices $\cA$ and $\cB$ stemming from the polar decomposition of $W$, \ie~
\begin{equation}\label{W.AB.dec}
W = \cB \cA, \quad \cB:=U_W |W|^\frac 12, \quad  \cA:=|W|^\frac 12.
\end{equation}

We further decompose $Q$ into a singular and regular part, namely $Q = L + M$. All these integral operators have explicit kernels; nonetheless, we display here in detail only the formula for $L$. After straightforward manipulations employing the formulas \eqref{H.res.ker} and \eqref{Psi.n.def}, we get 
\begin{equation}
(L \Psi)(x) =  
\cA(x) \Upsilon_2
\sum_{\sigma \in \{+,-\}}
\int_\Omega \langle \Psi_0^\sigma(y_1), \cB(y) \Psi(y) \rangle_{\C^4} \, \dd y \, \pd {\psi}_0^\sigma,
\end{equation}
where (with $\Upsilon$ as in \eqref{Ups.def})
\begin{equation}\label{Ups.def.ac}
\Upsilon_2= \Upsilon_2(\zeta_0(z))
=
I_2 \otimes  \Upsilon(\zeta_0(z)),
\quad \zeta_0(z) = \frac{z - \xi_0}{k_0(z)}, \quad k_0(z) = (z^2-\xi_0^2)^\frac12
\end{equation}
and 
\begin{equation}
\pd{\psi}_0^+=\frac{4 a^\frac 12 \sin \frac \Theta 4}{\Theta}
\begin{pmatrix}
1 \\ 0 \\ e^{-\ii \frac \Theta 2 } \\ 0
\end{pmatrix},
\qquad
\pd{\psi}_0^-=\frac{4 a^\frac 12 \sin \frac \Theta 4}{\Theta}
\begin{pmatrix}
0 \\ 1 \\ 0 \\ e^{-\ii \frac \Theta 2 }
\end{pmatrix}.
\end{equation}
For the regular part $M=M_1+M_2$, we have 
\begin{equation}\label{M12.def}
M_1 = A (H_0-z)^{-1} P_0 B - L, \qquad 
M_2 = A 
\Big(\sum_{n \in \Z\setminus \{0\}}
(H_n-z)^{-1} P_n \Big) B .
\end{equation}

\begin{lemma}\label{lem:M12.est}
Let $V$ be as in \eqref{V.def.ac} and $M_1$, $M_2$ be as in \eqref{M12.def}. Then
\begin{equation}
\|M_1\| = o(\|\Upsilon\|), \qquad \|M_2\| = \BigO(1), \quad z \to \pm \xi_0.
\end{equation}
\end{lemma}
\begin{proof}
The estimate of $\|M_1\|$ is almost the same as in the proof of Lemma~\ref{lem:M.HS}, we omit the details. 

To prove the estimate for $\|M_2\|$
let $f,g\in C_c^{\infty}(\Omega;\C^4)$, normalized in $L^2(\Omega;\C^4)$ and assume that $|z-|\xi_0||<1/2$. From formula~\eqref{Psi.n.def} it is straightforward to obtain the estimate
\begin{align}\label{bound on PnA}
\left(\int_{-a}^{a}
\|AP_n^{\sigma}f(x_1,x_2)\|_{\C^4}^2
\dd x_1\right)^{\frac 12}
\leq w(x_2)^{\frac 12}\left(\int_{-a}^a\|f(x_1,x_2)\|_{\C^4}^2\dd x_1\right)^{\frac 12}
\end{align}
for all $n\in\Z$, all $\sigma\in\{+,-\}$ and for almost all $x_2\in\R$, where
\begin{equation}
w(x_2):=\left(\int_{-a}^a\|W(x_1,x_2)\|^2\dd x_1\right)^{\frac 12}.
\end{equation}
An analogue of this inequality holds for $A,f$ replaced by $B,g$. By Schwarz's and Bessel's inequality, 
\begin{align}
&\sum_{|n|\geq 1}
\left|
\langle Af,(H_n-z)^{-1} P_n Bg\rangle\right|\\
&\leq 
\Big(\sum_{|n|\geq 1}
\||H_n-z|^{-\frac 12}P_n  Af\|^2\Big)^{\frac 12}
\Big(\sum_{|n|\geq 1}
\| |H_n-z|^{-1/2} P_n  Bg\|^2\Big)^{1/2}\\
&\leq \sup_{|n|\geq 1} \||AP_n|(-\partial_2^2+\xi_n^2)^{\frac 12}\otimes \pd{I_4} -z|^{-\frac 12}\|\,\|BP_n|(-\partial_2^2+\xi_n^2)^{\frac 12}\otimes \pd{I_4}-z|^{-\frac 12}\|
\\
&\leq 4 \sup_{|n|\geq 1} \|w^{1/2}|(-\partial_2^2+\xi_n^2)^{\frac 12}-z|^{-\frac 12}\|_{\mathfrak{S}^4(L^2(\R))}^2
\\
&\leq 4(2\pi)^{-\frac 14}\sup_{|n|\geq 1}\sup_{|z-|\xi_0||<1/2}\left(\int_{\R}|(\eta^2+\xi_n^2)^{\frac 12}-z|^{-2}\dd \eta\right)^{\frac 14}\|W\|_2,
\end{align}
where we used \eqref{bound on PnA} in the next-to-last inequality and the Kato-Seiler-Simon inequality, see~\cite[Thm.~4.1]{Simon-2005}, in the last inequality. The factor $4$ comes from estimating the matrix operator $H_n$ in terms of the scalar operator $(-\partial_2^2+\xi_n^2)^{\frac 12}$. The supremum over $n$ and $z$ in the final expression is finite.
\end{proof}

Having established suitable estimates of the regular part $M$, we prove the main result of this section.

\begin{theorem}\label{thm:ac.weak}
Let $V$ be as in \eqref{V.def.ac} and let $W= \Sigma V \Sigma^{-1}$ with $\Sigma$ as in \eqref{Sigma.def}. Define the matrices
\begin{equation}\label{U0.def.ac}
U_{ij}:= \sum_{k=1}^4 \sum_{\sigma \in \{+,-\}} \int_{\Omega} (\ov{\Psi_0^\sigma(x_1)})_k W_{kj}(x)  \, \dd x \, (\pd{\psi}_0^\sigma)_i.
\end{equation}
and
\begin{equation}
U^+:=
\begin{pmatrix}
U_{11} & U_{13}
\\
U_{31} & U_{33}
\end{pmatrix},
\qquad 
U^-:=
\begin{pmatrix}
U_{22} & U_{24}
\\
U_{42} & U_{44}
\end{pmatrix}.
\end{equation}

If $u^+ \in \C$ is a solution of $\det(u^+ U^+ + I_2)=0$ that satisfies $\Re u^+ >0$, then for any sufficiently small $\eps >0$, there exists an eigenvalue $z_+(\eps)$ of $D + \eps V$ satisfying
\begin{equation}\label{z+.exp.ac}
z_+(\eps) = -\xi_0 + \frac{\xi_0}{2(u^+)^2}  \eps^2 + o(\eps^2), \quad \eps \to 0+.
\end{equation}
Similarly, if $u^- \in \C$ is a solution of $\det(u^- U^- + I_2)=0$ that satisfies $\Re u^- <0$, then for any sufficiently small $\eps >0$, there exists an eigenvalue $z_-(\eps)$ of $D + \eps V$ satisfying
\begin{equation}\label{z-.exp.ac}
z_-(\eps) = \xi_0 - \frac{\xi_0}{2} (u^- )^2 \eps^2 + o(\eps^2), \quad \eps \to 0+.
\end{equation}
\end{theorem}
\begin{proof}
The strategy and individual steps are the same as in the one-dimensional case (Theorem~\ref{thm:1D.weak}) thus we indicate only the differences. Employing a decomposition as in \eqref{Q.decom} and the fact that the kernel of $L$ is separated, we convert the problem $-1 \in \spp(Q(z))$ into the algebraic equation 
\begin{equation}
\det \left(I_4 + \eps \Upsilon_2 (U+U_1)  \right)=0,
\end{equation}
with $U$ as in \eqref{U0.def.ac} and $U_1=\BigO\left(\frac{\eps\|M\|}{1-\eps \|M\|}\right)$. 
As an initial guess we consider 
\begin{equation}
\zeta^0_+= - \frac{2\ii}{\eps} u^+, \qquad 
\zeta^0_-= \frac{\ii}{2} \eps u^-,
\end{equation}
for which we obtain $\det \left( I_4 + \eps \Upsilon_2 U \right)= \BigO(\eps^2)$ as $\eps \to 0+$; the latter can be verified by the Laplace expansion of the determinant. The rest of the proof follows the lines of the one of Theorem~\ref{thm:1D.weak}, employing the estimates on $M$ from Lemma~\ref{lem:M12.est} and formulas \eqref{Ups.def.ac}.
\end{proof}

\begin{example}\label{Ex:V.diag.ac}
In particularly simple case where $V = \diag(v_1,v_2,v_3,v_4)$ with $v_j=v_j(x_2)$, $j=1,\dots,4$, and with $a = \Theta^2/(8 \sin^2(\Theta/4))$, straightforward calculations reveal that 
\begin{equation}
u^+=u^-=-\frac{1}{\int_{\R} \Tr(V(x_2)) \; \dd x_2}.
\end{equation}
Thus, depending on the sign of $\Re u^\pm$, we obtain eigenvalues $z_\pm$ obeying \eqref{z+.exp.ac} or \eqref{z-.exp.ac}.
\end{example}

{\footnotesize
	\bibliographystyle{acm}
	\bibliography{references}
}

\end{document}